\documentclass[reqno,10pt]{amsart}
\usepackage{amscd,amssymb,amsmath,color,url,stmaryrd}
\usepackage{latexsym}
\usepackage[all]{xy}
\usepackage{defs}
\usepackage[colorlinks=true]{hyperref}

\def\nBl{{\sqrt[n]{\rm Bl}}}
\def\NBl{{\sqrt[N]{\rm Bl}}}

\def\NN{{\bbN}}
\def\GG{{\bbG}}
\def\QQ{{\bbQ}}

\def\cProj{{\mathcal{P}roj}}

\def\nor{{\rm nor}}
\def\ord{{\rm ord}}

\def\mord{{\underline{\rm ord}}}

\def\hatcO{{\widehat\cO}}
\def\hatcI{{\widehat\cI}}
\def\hatcJ{{\widehat\cJ}}

\def\.{{,\dots,}}
\def\cK{{\calK}}

\def\cJ{{\calJ}}

\def\cR{{\calR}}
\def\cA{{\calA}}
\def\cI{{\calI}}
\def\cO{{\calO}}
\def\cC{{\calC}}
\def\cF{{\calF}}

\def\cD{{\calD}}

\def\cQ{{\calQ}}

\begin{document}
\title{Dream resolution and principalization I: enough derivations}
\author{Michael Temkin}

\thanks{This research is supported by ERC Consolidator Grant 770922 - BirNonArchGeom and BSF grant 2018193. The author is grateful to D. Abramovich, A. Belotto and J. W{\l}odarczyk for valuable discussions.}

\address{Einstein Institute of Mathematics\\
               The Hebrew University of Jerusalem\\
                Edmond J. Safra Campus, Giv'at Ram, Jerusalem, 91904, Israel}
\email{michael.temkin@mail.huji.ac.il}

\maketitle
\begin{abstract}
This is the first paper in a project on extending the dream principalization and resolution methods of \cite{ATW-weighted}, \cite{McQuillan} and \cite{Quek} to quasi-excellent, logarithmic and relative settings. We show that the main results of \cite{ATW-weighted} extend to regular schemes with enough derivations and are functorial with respect to all regular morphisms. This generality is already wide enough to formally imply that the same results hold in other categories, such as complex and $p$-adic analytic spaces. Our method has many common points with that of \cite{ATW-weighted}, but the accent is now shifted towards the study of weighted centers and their coordinate presentations. Not only we hope that this is a bit simpler and more conceptual, this method will be easily applied in the logarithmic and relative settings in the sequel.
\end{abstract}

\section{Introduction}

\subsection{Background and motivation}

\subsubsection{Some recent developments}
Much recent progress in resolution of singularities in characteristic zero is related to extending the classical setting in few aspects:

(0) {\em Quasi-excellent schemes and other categories.} Some (but not all) resolution results were extended to qe schemes by a black box method, e.g. see \cite{non-embedded}. First one establishes the case of formal varieties via deep algebraization results, such as Elkik's theory, and then induction on codimension (or a localization method which was used already by Hironaka) allows to deduce the general qe case at cost of producing a more complicated algorithm (even for varieties), see \cite[Section~4]{non-embedded}. In \cite{Temkin-survey} an alternative to the first step was suggested -- almost all classical proofs that work for varieties apply (with minimal changes) to arbitrary schemes with enough derivations. This covers formal varieties (as well as analytic spaces) and allows to eliminate the subtle algebraization step. Unfortunately, this approach was never worked out in detail, and the only paper where is was applied is \cite{ATW-relative}, where the context is much more general and complicated.

(1) {\em Logarithmic and relative settings.} One can work with log schemes instead of schemes, see \cite{ATW-principalization}, which allows to encode the boundary in a conceptual way and becomes absolutely critical in order to prove functorial semistable reduction (or relative resolution) theorems, see \cite{ATW-relative}.

(2) {\em Weighted algorithms.} One can allow basic (quasi) blowing up operations, which are more general than just blowing up smooth centers. Already in the case of log schemes, in addition to blowing up submonomial centers one has to allow certain weighted centers or root stack constructions, resulting in non-representable modifications. Log smooth functoriality cannot be achieved without this, see \cite{ATW-principalization}. However, the full strength of weighted centers was explored later in \cite{McQuillan} and \cite{ATW-weighted}, where a simple memoryless algorithm was constructed for schemes. Although the resolution is the same in the end, \cite{ATW-weighted} uses basic maximal contacts theory and only studies varieties, while \cite{McQuillan} uses a more explicit language of filtrations and applies to arbitrary qe schemes.

(3) {\em Combinations.} Weighted versions should exist in the general logarithmic and relative settings. So far, the first step was done in \cite{Quek} by constructing a weighted resolution and principalization for log varieties. In addition, a weighted algorithm for principalization on foliated varieties has been recently developed in \cite{ABTW}.

\begin{rem}
(i) A most natural way to try to achieve principalization/resolution (in various contexts) is to assign to an ideal/scheme an invariant and a canonical modification of a simple explicit form, such as a blowing up, so that the invariant drops after this operation. In particular, this leads to a memoryless algorithm which improves the singularity iteratively. It was a long journey to discover such algorithms, especially because they do not exist in the classical setting when the only allowed modifications are blowings up of smooth centers. Weighted blowings up provided the critical missing ingredient. The price one has to pay is that the technical setting becomes more complicated -- one has to consider stacks and allow a wider class of basic modifications. Not all community accepts this terminology so far, but I suggest to call the class of memoryless algorithms {\em dream algorithms}. They are the simplest ones from the algorithmic point of view, and because of this they allow maximal flexibility. 

(ii) As we remarked above, dream algorithms are expected to exist in the logarithmic and relative settings, and working this out is one of the main goals of this series of papers. Furthermore, in the foliated situation the setting is so complicated that a principalization method which only blows up smooth centers probably does not exist. Nevertheless, a dream foliated principalization does exist, as was shown very recently in \cite{ABTW}. Finally, extending a classical algorithm to the full generality of qe schemes is also a difficult task, which was not solved completely. The situation with the dream algorithm is better -- McQuillan develops such an algorithm in \cite{McQuillan}, and another description of a qe dream algorithms will be given in the sequel paper \cite{dream_qe}.
\end{rem}

\subsubsection{Goals of the project}
This paper originated in my joint project with G. Papas, whose original goal was to extend the results of \cite{ATW-relative} to singular base log schemes, including the standard log point and its non-reduced Kummer covers, and to establish a dream algorithm in this setting, thereby extending the results of \cite{Quek} to the largest generality. In this situation one cannot work with $\QQ$-ideals and usual blowings up and one of the main new ingredients we introduce in our work in progress are new types of modifications combining properties of usual blowings up and log blowings up. It became apparent that it will be convenient (if not necessary) to adjust the methods of \cite{ATW-weighted} and \cite{Quek}, rethink the notion of canonical centers, etc., so we decided to split the project into pieces. In addition, it turned out that the new method combined with some additional ideas can deal with the general qe case as well, so this was added as one of the goals of the project. Finally, there is also a goal of simplifying the method, popularizing it and making more accessible for a wider audience. Publishing our paper \cite{ATW-weighted} took four years and it was rejected thrice, which indicates that we failed to properly explain to the community the results, the methods and their novelty and significance.

For these reasons, we decided to split the current project into a series of papers so that each of them focuses on just a few aspects, and start with papers which do not use log geometry and involve a minimal amount of stack-theoretic issues. The current paper is devoted to extending the theory of canonical centers and dream principalization from \cite{ATW-weighted} to schemes (or stacks) with enough derivations. We reorganize proofs and terminology, but the main moving force is the same -- a basic maximal contact theory and derivations. Schemes with enough derivations provide the largest generality in which these methods work. In the second paper we will use these results and a certain descent and a spreading out procedures to generalize them to arbitrary qe schemes. This will involve new techniques -- Tschirnhaus coordinates and tubes in singular varieties. In the third paper we will extend the techniques of the first two to the logarithmic context, and the last paper of the project will be devoted to the most general case -- the relative case with an arbitrary base.

\subsection{Overview of the paper}

\subsubsection{Main results}
The main goal of this paper is to construct the theory of canonical centers associated to ideals on regular schemes with enough derivations. The main results are that such an ideal $\cI\subseteq\cO_X$ is contained in a unique weighted center $\cJ=[t_1^{d_1}\.t_n^{d_n}]$ such that the invariant $(\ud,V(\cJ))$ of $\cJ$ is maximal possible. Furthermore, this canonical center is compatible with any regular morphism $X'\to X$ and blowing it up with an appropriate regularizing root construction one obtains a principal transform $\cI'$ of $\cI$ whose canonical center has a smaller multiorder $\ud'<\ud$. These two results will be essentially used in the sequel paper \cite{dream_qe} in the proof of their generalization to qe schemes.

As an application one can easily deduce dream principalization and resolution algorithms on schemes and other geometric spaces with enough derivations. This is a standard reduction, which is only briefly outlined in the paper because these applications will not be used in \cite{dream_qe}, but will be subsumed there.

\subsubsection{The method}
At first glance the claim about existence of the maximal $\cI$-admissible weighted center looks very natural and one might expect that it should be easily accessible. Here are two warning marks: (1) this fails in positive characteristic, for example, due to an example by W{\l}odarczyk: if $X=\Spec(\bbF_2[x,y,z])$, then the Whitney umbrella $V(x^2+y^2z)$ has multiorder $(2,3,3)$ at each closed point of the $z$-axis and multiorder $(2,2)$ at the generic point, (2) this fails at non-excellent regular local rings (whose completion is not regular over them) because there is no resolution theory over them, see \cite[Example~2.1.5]{dream_qe} for a concrete example. This indicates that no robust elementary proof exists, and it is not surprising that derivations and a simplest form of the theory of maximal contact and coefficients ideals shows up. As usual, one achieves induction on dimension by restricting the coefficients ideal of $\cI$ onto a hypersurface of maximal contact, and the main issue is to prove independence of the maximal contact. The latter is done by coordinate changes which replace the maximal contact but keep the canonical center.

The approach of \cite{ATW-weighted} essentially uses homogenization of coefficients ideals, analogously to W{\l}odarczyk's approach in the classical resolution. The main simplification line of the current paper, which was also motivated by an idea of W{\l}odarczyk on Rees algebras of the canonical ideal, is to apply a sequence of operations which increase ideals, preserve inclusions but keep weighted centers unchanged. If we manage to increase $\cI$ by such operations getting a center $\cJ$, then the latter is automatically the canonical center of $\cI$. The first such operation just replaces $\cI$ with the coefficients ideal, the next one similarly increases the restriction of $\cI$ onto a maximal contact, etc. Thus, though the sequence of operations depends on the choice of maximal contacts the center we obtain in the end is canonical on the nose.

A complementary point of view is that the new method does most of the work with changes of coordinates on the level of centers only. In a sense the main point is to develop the theory of maximal contacts and coefficients ideals for weighted centers, which are ideals (or $\QQ$-ideals) of a very special shape. Studying them is a small nice problem which can be assigned as a good undergraduate project, and these theories just reduce to describing possible presentations of weighted centers in the form $\cJ=[\ut^\ud]$. So, to large extent the accent shifts to the elementary question of studying weighted centers themselves, and this theme will persist in the sequel papers in all other contexts -- qe, logarithmic, relative.

\subsubsection{Weighted centers}
We conclude the introduction with a brief overview of the paper. Section~\ref{tubsec} is devoted to studying the centers and blowings up along them. We decided to choose the formalism of $\QQ$-ideals defined by formal roots of integrally closed ideals, see \S\ref{Qsec}. This provides both a shortest route and a sufficient intuition. In \cite{dream_qe} we will even show that, in fact, a canonical center is always controllable by its rounding, which is an actual ideal (even though $d_i$ can be non-integral), but this is not essential in this paper, so we postpone studying this question. Our approach fits the definition of weighted blowings up in \S\ref{blowsec}, where we simply blow up the appropriate power of the ideal and then apply the normalized root stack construction along the exceptional divisor.

The main results of Section \ref{tubsec} are in \S\ref{tubesec}, where we study weighted centers on regular schemes and prove that the multiorder $\ud=(d_1\.d_n)$ is an invariant of a weighted center $\cJ=[\ut^\ud]$, any element $t\in\cJ$ of minimal order can be chosen as the first coordinate in a presentation of $\cJ$, and other coordinates can be lifted in any way from a presentation of the restricted center $\cJ|_{V(t)}$, see Theorem~\ref{centerindlem}. This theorem and Lemma~\ref{diffcenter} about derivations exhaust the theory of centers and their presentations we need.

\subsubsection{Canonical centers}
The theory of canonical centers is developed in Section \ref{canonicalsec}. We start with recalling generalities about derivations, maximal contacts and coefficients ideals in \S\ref{derivsec}, \S\ref{coefsec}. The only non-standard result is a simple key Lemma~\ref{restrlem} which expresses $\cI$-admissibility of  a center in terms of the coefficients ideal restricted to a maximal contact. Existence of the canonical center is deduced in Theorem~\ref{localcanonicalcenter} by a straightforward induction. In the same way, we prove in Theorem~\ref{dropth} that the multiorder drops after blowing up the canonical center. This time we also use Lemma~\ref{coefftransf} which provides a control on transform of derivations and coefficients ideals via a classical computation. We conclude the paper with a brief discussion of applications in \S\ref{appsec}.

Finally, let us explain the choice of terminology. Our definitions apply to arbitrary regular schemes rather than those with enough derivations. In particular, this is needed in order to include the general qe case that will be studied in \cite{dream_qe}. In such generality, a maximal $\cI$-admissible center does not have to exist and if exists it is not clear if it behaves well, commutes with localizations, etc. We by-pass this by localising the definition -- imposing the local maximality condition in the definition of canonical centers, which allows to glue local constructions on the nose. To some extent this is the matter of taste because a posteriori maximality and canonicity on schemes with enough derivations (or just excellent schemes) are equivalent, and this property is even compatible with arbitrary regular morphisms.

\subsection{Conventiones}
All schemes in this paper are assumed to be noetherian of characteristic zero and we only mention this assumption in the formulations of main results. We use underlines to denote tuples: $\ud=(d_1\.d_n)$, $\ut^\ud=(t_1^{d_1}\.t_n^{d_n})$, etc. For shortness, by an element of a sheaf we always mean a global section: $f\in\cO_X$ is a regular function on $X$, $\partial\in\cD_X$ is a global derivation $\cO_X\to\cO_X$, etc. Given an ideal $\cI\subseteq\cO_X$ on $X$ we denote the corresponding closed subscheme by $V_X(\cI)=\Spec_X(\cO_X/\cI)$ or simply by $V(\cI)$. When no confusion is possible this can also refer to the underlying closed set. The normalization of a reduced scheme $X$ will be denoted $X^\nor$. For an ideal $\cI$ on $X$ by $\cI^\nor$ we denote the integral closure of $\cI$. We will often use the notation $\hatX_x=\Spec(\hatcO_x)$ to denote the formal completion of a scheme $X$ at a point $x$, and then $\hatcI_x=\cI\hatcO_x$. Also, given an open $U\subseteq X$ we denote by $\cI_U=\cI|_U$ the restriction of $\cI$ on $U$.

\section{Weighted centers}\label{tubsec}
In this section we introduce and study weighted centers on regular schemes. They are very simple objects given by explicit equations. The only slightlyt subtle question is about the choice of a formalism because in general such centers are rational powers of ideals. We choose the formalism of $\QQ$-ideals and think that it is the most elementary one.

\subsection{Regularity}

\subsubsection{Regularity and parameters}\label{regsec}
A scheme $X$ is {\em regular} if all local rings are regular. A morphism is {\em regular} if it is flat and its fibers are geometrically regular. With the characteristic zero assumption, the word ``geometric'' can be omitted, absolute regularity just means regularity over $\QQ$ and formally locally such schemes are described as $\hatcO_x=k\llbracket\ut\rrbracket$, where $k=k(x)$ and $t_1\.t_n$ is a regular family of parameters. Furthermore, if $f\:Y\to X$ is a regular morphism taking $y\in Y$ to $x$, then formally locally $f$ can be described as $\hatcO_y=l\llbracket\ut,\us\rrbracket$, where $l=k(y)$.

\subsubsection{Parameters along a subscheme}
If $X$ is regular and $V\into X$ is a closed subscheme of codimension $d$, then by a {\em family of regular parameters along $V$} we mean $t_1\.t_d\in\cO_X$ such that $V=V_X(t_1\.t_d)$. Then automatically $V$ is regular and $t_1\.t_d$ is a partial family of regular parameters at any $x\in V$.

\subsubsection{Lifting parameters}\label{liftparam}
If $f\:Y\to X$ is a regular morphism and $X'\into X$ is a regular closed subscheme, then its pullback $Y'=X'\times_XY$ is regular. In particular, the pullback of a partial family of regular parameters on $X$ is again a partial family of regular parameters on $Y$.

\begin{lem}\label{completelem}
If $f\:Y\to X$ is a morphism of schemes of characteristic zero which is regular at $y\in Y$, and $x=f(y)$, then the homomorphism $\hatcO_x\to\hatcO_y$ is regular.
\end{lem}
\begin{proof}
Choose a regular family of parameters $t_1\.t_n\in\cO_x$ and elements $s_1\.s_m$, which restrict to a regular family of parameters of the fiber $\cO_y/m_x\cO_y$ of $f$ at $y$. Then $(\ut,\us)$ is a regular family of parameters of $\cO_y$ and hence also of its completion, and we have that $\hatcO_x=k(x)\llbracket\ut\rrbracket$ and $\hatcO_x=k(y)\llbracket\ut,\us\rrbracket$, making the claim obvious.
\end{proof}

Note that the claim of the lemma holds without the assumption on the characteristic, but the argument is a bit more involved. The opposite implication holds if the completion homomorphisms are regular, e.g. the schemes are qe, but fails in general.

\subsubsection{The order}
The {\em order of $\cI$ at a point $x\in X$} is the maximal number $d$ such that $\cI_x\subseteq m_x^d$. By convention, $\ord(0)=\infty$. We define the global order by $\ord_X(\cI)=\max_{x\in X}\ord_x(\cI)$ and the {\em order function} $\ord_\cI:X\to\NN\cup\{\infty\}$ which sends $x$ to $\ord_x(\cI)$.

\begin{lem}\label{logordlem}
Let $x$ be a regular scheme, $\cI$ an ideal, and $x\in X$ a point with $\hatcO_x=k\llbracket t_1\.t_n\rrbracket$. Then

(i) $\ord_x(\cI)=\ord_x(\cI_x)=\ord_x(\hatcI_x)$ and this is the minimal number such that there exists $f\in\cI$ whose decomposition in $\hatcO_x$ contains a non-zero monomial $c\ut^\ud$ of degree $d=d_1+\dots+d_n$.

(ii) Functoriality: if $f\:X'\to X$ is a regular morphism, then $\ord_{\cI'}=\ord_\cI\circ f$.
\end{lem}
\begin{proof}
The first claim is obvious and the second one easily follows from the formal local description of regular morphism in \S\ref{regsec} and Lemma~\ref{completelem}.
\end{proof}

\subsection{$\QQ$-ideals}\label{Qsec}

\subsubsection{The definition}
Informally speaking, a $\QQ$-ideal on a normal scheme $X$ is a rational power of an ideal. There are a few equivalent ways to formalize this, and we choose the following one: a {\em $\QQ$-ideal} is a formal rational power $[I]^a$, where $a\in\QQ_{\ge 0}$ and $[I]$ is the equivalence class of ideals that have the same integral closure. In particular, one can simply identify the equivalence class $[I]$ with the integral closure $I^\nor$. By definition $[I]^a=[J]^b$ if $[I^{an}]=[J^{bn}]$ for $n$ such that $an,bn\in \NN$. In particular, $[I]^{m/n}=[I^m]^{1/n}$ so one can always choose $a$ of the form $1/n$. Multiplication, summation and inclusion of ideals are compatible with integral closures, hence descend to the set of equivalence classes and then extend to $\QQ$-ideals as follows: $[I]^{1/n}+[J]^{1/m}=[I^m+J^n]^{1/mn}$, $[I]^{1/n}[J]^{1/m}=[I^mJ^n]^{1/mn}$ and $\cI\subseteq\cJ$ if and only if $\cI+\cJ=\cJ$.

Any ideal $I$ can be viewed as a $\QQ$-ideal $[I]$. In particular, we will use the notation $[t_1\.t_m]$ to denote the $\QQ$-ideal generated by $t_1\.t_m\in\cO_X$, and, more generally, use the notation $[t_1^{a_i}\.t_m^{a_m}]$ with $a_i\in\frac{1}{n}\NN$ to denote the ideal $[t_1^{na_1}\.t_m^{na_m}]^{1/n}$. Of course, this is consistent with the usual notations for ideals, namely, $[t_1^{a_i}\.t_m^{a_m}]$ is indeed the sum of ideals $[t_i^{a_i}]=[t_i]^{a_i}$. Conversely, any $\QQ$-ideal $\cI$ can be {\em rounded} to an ideal $(\cI)=\cI\cap\cO_X$, which is the maximal ordinary ideal contained in $\cJ$. Of course $(\cJ)$ is integrally closed. We say that a $\QQ$-ideal $I$ is an ordinary ideal if $I=[J]$ for an ideal $J$, and this happens if and only if $I=[(I)]$, that is, $I$ is generated by its rounding.

\begin{rem}
(i) In fact, $\QQ$-ideals just provide a formalization of objects widely used in resolution of singularities: Hironaka's idealistic exponents or marked ideals $(I,a)$. The operations on them match the operations on marked ideals. Also, such objects were considered in commutative algebra for completely different motivation.

(ii) It is easy to see that passing to the equivalence classes of ideals one obtains a cancellative monoid and the monoid of $\QQ$-ideals is its divisible hull.

(iii) One can also realize $\QQ$-ideals as certain ideals in the topology generated by blowings up and finite covers or just in the $h$-topology. In particular, passing to the $h$-topology identifies all ideals with the same integral closure and makes the monoid of ideals divisible, though there are $h$-ideals which are not $\QQ$-ideals. We will not use this approach, but it may serve as an additional source of intuition. In particular, it conceptually explains the meaning of operations on $\QQ$-ideals (just sum and product of $h$-ideals), and the rounding $(\cI)$ of a $\QQ$-ideal $\cI$ is just the pushforward of $\cI$ under the restriction of sites $X_h\to X_{\rm Zar}$.
\end{rem}

\subsubsection{Support of a $\QQ$-ideal}
Of course one can formalize the vanishing locus $V(\cJ)$ as a sort of $\QQ$-subscheme, and this is a good source of intuition, e.g. $V(\cJ)\into V(\cJ')$ if and only if $\cJ\subseteq\cJ'$, but we will only need the set-theoretic version. So we just set $V(\cJ)=|V(\cI)|$, where $\cI$ is any representative of any power $\cJ^N$, which is an ideal.

\subsubsection{Pullbacks}
If $X'\to X$ is a morphism with a normal source, then the equivalence is preserved by pulling back ideals, hence we can define the pullback of a $\QQ$-ideal $\cJ=[\cI]^a$ on $X$ to be $\cJ'=[\cI\cO_{X'}]^a$. Naturally, we will use the suggestive formal notation $\cJ'=\cJ\cO_{X'}$. Pullbacks are compatible with taking positive rational powers of $\QQ$-ideals.

\subsubsection{Valuation functions}
Let $\cI=[I]^d$ be a $\QQ$-ideal. We will use the suggestive notation $x^a\in\cI$ if $[x]^a\subseteq\cI$. This happens if and only if $x^{na}\in (I^{nd})^\nor$ for some, and then any, $n$ such that $na$ and $nd$ are integral. Now we define the {\em valuation function} induced by $\cI$ as follows: $\nu_\cI(f)$ is the supremum of the set of all numbers $a\in\QQ_{\ge 0}$ such that $f\in\cI^a$. In particular, $(\cI)$ is just the set of elements $f\in\cO_X$ of valuation at least $1$.

By convention, $\nu_\cI(f)=\infty$ if $f\in\cI^a$ for any $a$ and the set of all such elements $f$ is an ideal called the kernel of $\nu_\cI$. In fact, $\nu_\cI(f)=\infty$ if and only if $f$ vanishes on each connected component $X'\subseteq X$ such that $\cJ|_{X'}\neq\cO_{X'}$.

\subsubsection{Local valuation function}
By $\nu_{\cI,x}(f):=\nu_{\cI_x}(f)$ we denote the {\em local valuation} of $f$ at a point $x$. Clearly, $\nu_\cI(f)=\min_{x\in X}\nu_{\cI,x}(f)$.

\subsubsection{Properties of valuation functions}
The valuation function is power multiplicative in both arguments: $\nu_\cI(f^a)=a\nu_\cI(f)$ and $\nu_{\cI^a}=a^{-1}\nu_\cI$. Also, it satisfies the strong triangle inequality $\nu_\cI(a+b)\ge\min(\nu_\cI(a),\nu_\cI(b))$ and is submultiplicative: $\nu_\cI(ab)\ge\nu_\cI(a)+\nu_\cI(b)$. In particular, $\nu_\cI$ is a usual semivaluation (resp. a valuation) if the latter inequality is an equality (resp. and the kernel is trivial).

\subsection{Weighted centers}\label{tubesec}

\subsubsection{Centers and presentations}
By a {\em weighted center} or simply a {\em center} on a regular scheme $X$ we mean a $\QQ$-ideal $\cJ$ such that there exists a (non-strictly) increasing tuple $\ud=(d_1\.d_n)$ of positive rational numbers such that $X$ possesses a finite open cover $X=\cup_iX_i$ and each $\cJ_i=\cJ|_{X_i}$ admits a presentation of the form $\cJ_i=[t_1^{d_1}\.t_n^{d_n}]=[\ut^\ud]$, where $t_1\.t_n\in\cO_{X_i}$ is a family of regular parameters along $V(\cJ_i)$. In particular, the support $V(\cJ)$ of a center is a regular scheme. A presentation $\cJ=[\ut^\ud]$, if exists, is called a {\em (global) presentation} of $\cJ$.

By definition, the case of $n=0$ corresponds to the ideal of a union of connected components of $X$ (i.e. an ideal which is zero along its support). For functoriality reasons, the trivial ideal $\cO_X$ is also considered to be a center, which does not admit a presentation. We call it the {\em trivial center}.

\begin{rem}
(i) The definition of a weighted center is non-local in principle because it postulates existence of a global multiorder $\ud$ (whose uniqueness will be proved later). In particular, disjoint union of centers of different multiorders is not a center.

(ii) Centers are preserved by pullbacks under regular morphisms because partial families of parameters are preserved by such pulbacks.
\end{rem}

\subsubsection{Weighted valuation functions}
It will be very useful in the sequel to have an explicit formally-local description of $\nu_\cJ$.

\begin{lem}\label{vallem}
Let $X$ be a regular scheme, $\cJ=[\ut^\ud]$ a center with a global presentation, $\ud=(d_1\.d_n)$ and $w_i=d_i^{-1}$ for $1\le i\le n$. For a point $x\in V(\cJ)$ fix an isomorphism $C\llbracket t_1\.t_n\rrbracket=\hatcO_{X,x}$, where $C=k(x)\llbracket t_{n+1}\.t_m\rrbracket$. Then $\nu_{\cJ,x}$ is a valuation on $\cO_{X,x}$ which extends to the valuation $\nu_\hatcJ$ on $\hatcO_{X,x}$ and the latter can be described as follows: $$\nu_\hatcJ\left(\sum_{\ua\in\NN^n} c_\ua t_1^{a_1}\dots t_n^{a_n}\right)=\min_{\ua\in\NN^n|\ c_\ua\neq 0}\left(a_1w_1+\dots+a_nw_n\right).$$
\end{lem}
\begin{proof}
The formula on the right defines a function $\widehat{\nu}'$ on $C\llbracket\ut\rrbracket=\hatcO_x$, which is linear on monomials, and hence is easily seen to be a valuation (a so-called generalized Gauss valuation). Therefore its restriction $\nu'$ onto $\cO_x$ is also a valuation and we should prove that it coincides with $\nu_{\cJ,x}$. Since $\nu'$ is a valuation, for any $c\in\QQ_{\ge 0}$ the inequality $\nu'(f)\ge c$ defines an integrally closed ideal $\cI_c\subseteq\cO_x$. Moreover,  $\nu'$ is the minimal valuation such that $\nu'(t_i)\ge w_i$ for $1\le i\le n$, and hence $\cI_c=(\ut^{c\ud})^\nor=[\cJ^c]$ is the minimal integrally closed ideal containing $\ut^{c\ud}$ for any $c$ such that $c\ud\in\NN^n$. Unravelling the definition of $\nu_{\cJ,x}$ we obtain that it coincides with $\nu'$, completing the proof.
\end{proof}

\subsubsection{Multiorder}
Let $\cJ=[\ut^\ud]$ be a center. If $\cJ=\cO_X$ is trivial its order is set to be 0 and the multiorder is $(0)$. Otherwise, $d_1=\ord(\cJ)$ is called the {\em order} of $\cJ$ and the tuple $(\ud)=(d_1\.d_n,\infty)=\mord(\cJ)$ is called the {\em multiorder} of the presentation. In particular, $\mord(\cJ)=(\infty)$ if and only if $n=0$. Recall also that we always assume that $d_1\le d_2\le\dots\le d_n$. Of course, the multiorder depends only on $\cJ$ so the notation is unambiguous, but we postpone the proof until Corollary~\ref{invarcenter}. Thus the set of possible multiorders consists of $(0)$ and all increasing tuples in $\coprod_{n\ge 0}(\QQ_{>0}^n\times\{\infty\})$. We order this set lexicographically and denote it $\cQ$.

\begin{rem}
An alternative is to consider tuples without the infinity and order them lexicographically with the convention that an initial subtuple of a tuple is larger. This is the convention chosen in \cite{ATW-weighted}. We will usually use this convention in the notation and omit $\infty$ when this cannot cause to confusions.
\end{rem}

\subsubsection{The invariant}
The {\em invariant} $\inv(\cJ)=(\ud,V(\cJ))$ of $\cJ$ is combined from the multiorder and the support. We partially order the set of invariants lexicographically: $(\ud',V(\cJ'))\le(\ud,V(\cJ))$ if either $\ud'<\ud$ or $\ud'=\ud$ and $V(\cJ')\subseteq V(\cJ)$. Here is a first application of this notion.

\begin{lem}\label{subcenter}
Let $X$ be a regular scheme with nested weighted centers $\cJ\subseteq\cJ'$ which locally possess presentations of the same multiorder $\ud$. Then $\cJ'=\cJ$ if and only if $V(\cJ)=V(\cJ')$.
\end{lem}
\begin{proof}
Raising $\cJ$ and $\cJ'$ to a large enough power we can assume that the multiorders are natural and hence both are ordinary integrally closed ideals. It suffices to prove that if $x\in V(\cJ)\cap V(\cJ')$, then $\cJ_x=\cJ'_x$. Locally formally the isomorphism class of a center at $x$ depends only on the multiorder of a presentation, the residue field $k(x)$ and the dimension $\dim_x(X)$. This implies that the HS functions of $A=\cO_x/\cJ_x$ and $A'=\cO_x/\cJ'_x$ coincide, and hence the surjection $A\onto A'$ is an equality.
\end{proof}

\subsubsection{Maximal contacts to a center}
Now, we will show how to construct presentations inductively. This can be viewed as a maximal contact theory for centers. In a sense, the situation is the best possible (and natural) one: one can start with picking a coordinate $t=t_1$ of maximal possible valuation function, then the restriction $\cJ_H=\cJ|_H$ onto $H=V(t)$ is a center, and one can simply combine $t_1$ with any lift of a presentation of $\cJ_H$ to $X$.

\begin{theor}\label{centerindlem}
Let $\cJ$ be a center on $X$ of order $d$ and assume that $H=V(t)$ is a regular subscheme containing $V(\cJ)$. Then

(i) $\nu_\cJ(t)\le w=d^{-1}$ and the equality holds if and only if locally at any point $x\in V(\cJ)$ there exists a family $t_1=t,t_2\.t_n$ of parameters along $V(\cJ)$ which gives rise to a presentation $\cJ=[\ut^\ud]$.

(ii) If $\nu_\cJ(t)=w$, then $\cJ_H=\cJ|_H$ is a center on $H$, and a family $t_1=t,t_2\.t_n$ of parameters along $V(\cJ)$ gives rise to a presentation $\cJ=[\ut^\ud]$ if and only if the restrictions $\ot_i=t_i\cO_H$ give rise to a presentation $\cJ_H=[\ot_2^{d_2}\.\ot_n^{d_n}]$.
\end{theor}
\begin{proof}
All claims can be checked locally at a point $x\in V(\cJ)$, so we can assume that there exits a presentation $\cJ=[\ut^\ud]$, where $d_1=d$.

(i) Since $t$ is a parameter at $x$, its formal decomposition with respect to $\ut$ involves a monomial of order one, and then Lemma~\ref{vallem} implies that $\nu_{\cJ,x}(t)\le\max_i(d_i^{-1})=w$. If $\cJ=[\ut^\ud]$ with $t=t_1$, then $d=\ord(\cJ)=d_1$ and $\nu_\cJ(t)=w$. Conversely, assume that $\nu_\cJ(t)=w$. Then Lemma~\ref{vallem} implies that the formal decomposition of $t$ with respect to a full system of regular parameters $t_1\.t_m$ extending $\ut$ involves a non-zero linear term $at_i$ with $d_i=d$. Switching $t_1$ and $t_i$ we can assume that $i=1$ and then $t,t_2\.t_n$ is another system of regular parameters. It follows that we have an inclusion of centers $[t^d,t^{d_2}_2\.t^{d_{n}}_{n}]\subseteq\cJ$, which has to be an equality by Lemma~\ref{subcenter}.

(ii) By (i) locally at $x$ there exists a presentation $\cJ=[\ut^\ud]$ and then $\cJ_H=[\ot^{d_2}_2\.\ot^{d_n}_n]$ is a center on $H$. It remains to prove the lifting claim. So assume that $\ut'=(t,t'_2\.t'_n)$ is a family of parameters along $V(\cJ)$ such that $\cJ_H=[\ot'^{d_2}_2\.\ot'^{d_n}_n]$. It suffices to check that $\nu_\cJ(t'_i)\ge w_i=d_i^{-1}$ for $2\le i\le n$, because then $t'^{d_i}_i\in\cJ$ and hence $[\ut'^\ud]=\cJ$ by another use of Lemma~\ref{subcenter}.

By Lemma~\ref{vallem} $\nu_{\cJ,x}(t'_i)$ is computed from the presentation $t'_i=\sum_{\ua}c_{\ua}\ut^\ua$ as the minimum of the values of $a_1w_1+\dots +a_nw_n$ on non-zero monomials of the presentation. Reducing this modulo $t$ and applying the same lemma to the obtained presentation of $\ot'_i$ we obtain that $w_i=\nu_{\cJ_H,x}(\ot'_i)$ is the minimum of the expressions $a_2w_2+\dots +a_nw_n$ on the monomials which do not vanish on $H$ and hence do not involve $t_1$. Combining this with the case when $a_1>0$ and hence $a_1w_1+\dots+a_nw_n\ge w_1\ge w_i$, we obtain that $\nu_\cJ(t'_i)\ge w_i$, as required.
\end{proof}

Any $t$ and $H(t)$ as in the lemma are called {\em maximal contact} parameter and hypersurface of $\cJ$. They exist locally whenever $0<\ord(\cJ)<\infty$.

\begin{cor}\label{invarcenter}
Let $X$ be a regular scheme with a weighted center $\cJ=[t_1^{d_1}\.t_n^{d_n}]$.

(i) If $\cJ\subseteq\cJ'=[t'^{d'_1}_1\.t'^{d'_{n'}}_{n'}]$ is inclusion of centers, then $\ud'\le\ud$ and $V(\cJ')\subseteq V(\cJ)$, and both equalities hold if and only if $\cJ=\cJ'$. In particular, $\mord(\cJ)$ and $\inv(\cJ)$ are independent of the choice of parameters and hence depend only on $\cJ$.

(ii) Assume that $x\in V(\cJ)$ and $\ut'$ is another regular family of parameters at $x$. Then the following conditions are equivalent: (a) $\ut'$ gives rise to another presentation of $\cJ$ at $x$, and hence $\cJ_x=[\ut'^\ud]$, (b) $t'^{d_i}_i\in\cJ_x$ for $1\le i\le n$, (c) $\nu_{\cJ,x}(t'_i)\ge 1/d_i$ for $1\le i\le n$.
\end{cor}
\begin{proof}
Raising $\cJ$ and $\cJ'$ to a large enough power we can assume that the multiorders are natural and hence both are ordinary integrally closed ideals. The second claim of (i) was already proved in Lemma~\ref{subcenter}. The inclusion of supports is obvious, so we should prove that $\ud'\le\ud$. We will use induction on $n'$ with the case $n'=0$ being trivial. We can work locally at a point $x\in V(\cJ')$. If $d'_1<d_1$ we are done, so assume that $d'_1\ge d_1$. Then the equality holds because $d'_1=\ord(\cJ')\le\ord(\cJ)=d_1$. Set $d=d_1$ and $t=t_1$ for shortness. Since $t^d\in\cJ\subseteq\cJ'$ we have that $\nu_{\cJ'}(t)\ge 1/d$ and then by Theorem~\ref{centerindlem}(i) we can choose another presentation of $\cJ'$ so that $t=t'$. Furthermore, if $H=V(t_1)$, then by Theorem~\ref{centerindlem}(ii) $\cJ|_H\subseteq\cJ'|_H$ are centers with presentations of multiorders obtained by omitting the first entry. So, $(d_2\.d_n)\le(d'_2\.d'_{n'})$ by the induction assumption, and hence $\ud\le\ud'$.

The equivalence of (b) and (c) in (ii) follows from the power multiplicativity of $\nu_{\cJ,x}$. Condition (b) means that $\cJ'=[\ut'^\ud]$ is contained in $\cJ_x$, and by the second claim of (i) this happens if and only if (a) holds.
\end{proof}

\subsubsection{Differential calculus of centers}
We finish this section with a reinterpretation of the maximal contact theory in terms of derivations. This provides a relation to the classical theory of maximal contact and will be essential in order to prove existence of the canonical center of an ideal.

We will recall the definition of derivations of ideals in \S\ref{notationsec} below, but it seems more natural to include the following lemma in this section. Also, we will only derive roundings of the centers, but for simplicity we use the notation $\cD^{(\le i)}_X(\cJ)$ instead of $\cD^{(\le i)}_X((\cJ))$. Similarly, the {\em coefficients ideal} $\cC_X(\cJ)=\sum_{i=0}^{d-1}(\cD_X^{(\le i)}(\cJ))^{\frac{d!}{d-i}}$ depends only on the rounding.

\begin{lem}\label{diffcenter}
Let $\cJ$ be a center of an integral order $d$, then

(i) If $t$ is a parameter along $V(\cJ)$ and $t\in\cD_X^{(\le d-1)}(\cJ)$, then $t$ is a maximal contact to $\cJ$.

(ii) $\cD_X^{(\le i)}(\cJ)\subseteq\cJ^{\frac{d-i}d}$ for $0\le i\le d-1$, and $\cC_X(\cJ)\subseteq\cJ^{(d-1)!}$.

(iii) Assume that $t$ is a maximal contact to $\cJ$, $H=V(t$), $f\in\cO_X$ and $\partial\in\cD_X$ is a derivation such that $\partial(t)\in\cO_X^\times$. Then $f\in\cJ$ if and only if $\partial^i(f)|_H\in\cJ^{\frac{d-i}d}|_H$ whenever $0\le i\le d-1$.
\end{lem}
\begin{proof}
We can work locally at a point $x\in V(\cJ)$. In particular, we can assume that $\cJ=[\ut^\ud]$ with $d_1=d$. It follows from Lemma~\ref{vallem} that for any $f\in\hatcO_x$ one has that $\nu_\hatcJ(\partial(f))\ge\nu_\hatcJ(f)-w$, where $w=d^{-1}$. Indeed, the lemma implies that it suffices to check this for a monomial $m$, and since for any derivation $\nabla$ the element $\nabla(m)$ is of the form $\sum_i a_im/t_i$ we obtain that $$\nu_\hatcJ(\nabla(m))\ge\nu_\hatcJ(m)-\max_i(\nu_\hatcJ(t_i))=\nu_\hatcJ(m)-w.$$

By induction on $i$ we now obtain that $\nu_\hatcJ(f)\ge \frac{d-i}d$ for any $f\in\cD_X^{(\le i)}(\cJ)$ yielding the first claim of (ii), and the second claim follows immediately via the formula for $\cC_X(\cJ)$. In addition, if $t$ is as in (i), then $\nu_\hatcJ(t)\ge w$ and hence $t$ is a maximal contact by Theorem~\ref{centerindlem}.

Finally, let us prove (iii). The direct implication follows from (ii), so let us prove the inverse one. Let $\hatcO_{X,x}$ be the $(t)$-adic completion. Note that $\partial$ extends to (any) formal completion by continuity, and by a classical result, that we recall in Lemma~\ref{taylor} below, $\Ker(\partial)=\cO_{H,}$ yielding an isomorphism $\hatcO_{X,x}=\cO_{H,x}\llbracket t\rrbracket$. In particular, for any $f\in\hatcO_{X,x}$ we view the restriction $\of\in\hatcO_{H,x}$ as an element of $\Ker(\partial)$ or as the free term $\of_0$ of the Taylor expansion $f=\sum_{i=0}^\infty\of_it^i$. By Theorem~\ref{centerindlem}, $\hatcJ=\cJ\hatcO_{X,x}$ possesses a presentation $[\ut^\ud]$ with $t_1=t$ and $t_2\.t_n\in\Ker(\partial)$, hence $\cJ_H:=\cJ|_H$ lies in $\cJ$ once one identifies $\cO_{H,x}$ with $\Ker(\partial)$. Now, $i!\of_i=\partial^i(f)|_H\in\cJ_H^{\frac{d-i}d}\subset\cJ^{\frac{d-i}d}$ for any $i<d$. In addition, $t^i\in\cJ^{\frac{i}d}$ because $t$ is a maximal contact, and hence $f\in\cJ$.
\end{proof}

We have used above the following splitting lemma (which is the induction step in a splitting theorem of Nagata-Zariski-Lipman, see \cite[Theorem~30.1]{Matsumura-ringtheory}).

\begin{lem}\label{taylor}
Let $A$ be a ring of characteristic zero, $t\in A$ an element such that $A$ is $t$-adically complete and $\oA=A/tA$. If $\partial\:A\to A$ is a derivation such that $\partial(t)$ is a unit, then $\Ker(\partial)\toisom\oA$ and this gives rise to an isomorphism $\oA\llbracket t\rrbracket\toisom A$. In particular, $\oA\llbracket t\rrbracket\toisom A$ if and only if there exists a derivation $\partial\:A\to A$ such that $\partial(t)$ is a unit.
\end{lem}
\begin{proof}
Set $B=\Ker(\partial)$. Replacing $\partial$ by $\partial(t)^{-1}\partial$ we can assume that $\partial(t)=1$. The homomorphism $B\llbracket t\rrbracket\to A$ is injective because if $\partial(b)=0$ for $b=\sum_{i=0}^n b_it^i\in B\llbracket t\rrbracket$, then $\sum_{i=1}^\infty ib_it^{i-1}$ vanishes, and hence $b\in B$. Thus, it suffices to prove that the composed homomorphism $B\into A\onto\oA$ is surjective. This follows from the observation that for any $a\in A$ the element $\sum_{i=0}^\infty \frac{(-1)^i}{i!}t^i\partial^i(a)$ lies in $B$ and has the same image in $\oA$ as $a$.
\end{proof}

\subsection{Blowings up of centers}\label{blowsec}
We will only need normalized constructions, so we mainly discuss them. We do not prove anything new in this subsection, and the construction of weighted blowings up was described, for example, in \cite[\S3]{ATW-weighted}, while a thorough study of the subject can be found in \cite[\S3]{Quek-Rydh}. So we choose the fastest root, though with a small change of accents and notation/terminology.

\subsubsection{Normalized root stacks}
Root stacks were introduced in \cite[\S2]{Cadman} and \cite[Appendix B]{AGV}. If $X$ is a normal Artin stack with a Cartier divisor $D$, then the {\em root stack} $X'=X[\sqrt[n]{D}]$ is the fiber product $X\times_\cA\cA$, where $\cA=[\bbA^1/\GG_m]$, the morphism $X\to\cA$ is the universal morphism induced by $D$ (so that $D$ is the pullback of the universal Cartier divisor at the origin of $\cA$) and the morphism $\cA\to\cA$ is the $n$-th power morphism. By $\sqrt[n]{D}$ we denote the pullback of the origin under the projection $X'\to\cA$ (which is the base change of $X\to\cA$), so it is  an $n$-th root of $D$. In general, the morphism $X'=X[\sqrt[n]{D}]\to X$ is non-representable even when $X$ is a scheme, and $(X',D')$ is the universal $X$-stack with a fixed $n$-th root of $D$: if $g\:Y\to X$ is another morphism and $D_Y$ is such that $D_Y^n=g^*(D)$ then there exists a factorization $Y\to X[\sqrt[n]{D}]\to X$ such that $\sqrt[n]{D}$ pulls back to $D_Y$, and it is unique up to a unique 2-isomorphism, see \cite[Definition~2.2]{Cadman} or \cite[Appendix B.2]{AGV}

The normalization $(X\times_\cA\cA)^\nor$ is called the {\em normalized root stack}. Since roots of Cartier divisors are unique on normal schemes, it is characterized by the even simpler property of being the universal normal $X$-stack such that the pullback of $D$ is an $n$-th power.

\begin{exam}
If $D=(t^2)$ is a square, then $X[\sqrt{D}]^\nor=X$, while $X[\sqrt{D}]$ is usually non-normal and with a non-trivial stacky structure. This is the case already when $X=\Spec(k[t])$ and $X[\sqrt{D}]=[\Spec(k[t,x]/(t^2-x^2))/\mu_2]$ where $\mu_2$ switches $x$ and $t$.
\end{exam}

\subsubsection{Normalized root blowings up}
Given a stack $X$ with an ideal $\cJ$ by the {\em normalized root blowing up} $\nBl_\cJ(X)^\nor\to X$ we mean the composition $r\circ f$ of the usual blowing up $f\:X'=\Bl_\cJ(X)\to X$ along $\cJ$ and the normalized root stack construction $r\:X'[\sqrt[n]{\cJ'}]^\nor\to X'$, where $\cJ'=\cJ\cO_{X'}$ is the exceptional divisor of $f$. Combining the universal properties of all these constructions one immediately obtains that $Y=\nBl_\cJ(X)^\nor$ is the universal normal $X$-stack such that the pullback of $\cJ$ is an $n$-th power of an invertible ideal.

\subsubsection{Normalized blowings up of $\QQ$-ideals}
The universal property implies that $\nBl_\cJ(X)^\nor=\sqrt[mn]{\Bl}_{\cJ^m}(X)^\nor$ for any $m$, and hence the construction extends to the case when $\cJ$ is a $\QQ$-ideal via the same formula $\nBl_\cJ(X)^\nor=\sqrt[mn]{\Bl}_{\cJ^m}(X)^\nor$. Moreover, it follows that $\nBl_\cJ(X)^\nor$ is the universal normal $X$-stack such that the pullback of $\cJ$ is an $n$-th power of an invertible ideal and $\nBl_\cJ(X)^\nor=\sqrt[mn]{\Bl}_{\cJ^m}(X)^\nor$ for any $m,n\ge 1$. Thus, any normalized root blowing up can be expressed as a normalized blowing up of a $\QQ$-ideal: $\nBl_\cJ(X)^\nor=\Bl_{\cJ^{1/n}}(X)^\nor$, but it will be convenient to play with roots both of the ideal and of the blowing up.

Since we only define and consider normalized blowings up along $\QQ$-ideals we will skip noramlization from the notation and simply write $\Bl_\cJ(X)$ instead.

\subsubsection{Charts and blowings up}
If $\cJ=[x_1^{a_1}\.x_r^{a_r}]$ with $a_i\in\frac{1}{m}\NN$, one can explicitly describe $X'=\Bl_{\cJ}(X)$ via charts. Set $y_i=x_i^{ma_i}\in\cO_X$, then $Y=\Bl_{\cJ^m}(X)$ is covered by the charts $Y_i=\Spec_X(\cO_X[\frac{y_1}{y_i}\.\frac{y_r}{y_i}])$ and $X'$ is covered by the charts $X'_i=Y_i[y_i^{1/m}]^\nor$. Here is the only case we will be interested in -- blowing up of a center of a special form:

\begin{exam}\label{chartexam}
Assume that $X$ is regular and $\cJ=[t_1^{1/w_1}\.t_n^{1/w_n}]$ is a center with a global presentation with $\uw\in\NN^n$. Then $X'=\Bl_\cJ(X)$ is covered by the charts $X'_i$ corresponding to $t_i^{1/w_i}$ and one has that $$X'_i=\left[\Spec_X\left(\cO_X\left[s_i=t_i^{1/w_i},\frac{t_1}{s_i^{w_1}}\.\frac{t_n}{s_i^{w_n}}\right]\right)/\mu_{w_i}\right].$$ Indeed, it follows from the above description that $X'_i$ is the normalization of the righthand side. The latter is easily seen to be a regular stack, hence the normalization morphism is an isomorphism and $X'$ is regular.
\end{exam}

As a corollary we obtain the following result, which provides a supply of stack-theoretic modifications we will be using in dream principalization.

\begin{lem}\label{regblow}
Assume that $X$ is regular, $\cJ$ is a center of multiorder $(d_1\.d_n)$ and $N>0$ a natural number such that $w_i=N/d_i\in\NN$ for $1\le i\le n$. Then $\NBl_\cJ(X)$ is a regular scheme.
\end{lem}
\begin{proof}
This is covered by the above example applied to the center $\cJ^{1/N}$.
\end{proof}

\begin{rem}
The lemma can be interpreted as the claim that any blowing up of a weighted center can be desingularized by a sufficiently large root construction along the exceptional divisor. Although we will really want to blow up canonical centers with $d_i^{-1}\notin\NN$, in order to keep regularity we will have to refine this construction by extracting appropriate roots.
\end{rem}

\subsubsection{Strict transforms}
As usual, by a strict transform of a closed subscheme $Z$ under a (normalized) root blowing up $X'\to X$ along $\cJ$ one means the schematic closure $Z'$ of $Z\setminus V(\cJ)$ in $X'$. We will only need the following very particular case which is easily checked using the charts.

\begin{lem}\label{strictlem}
Assume that $X$ is regular, $\cJ$ is a center with a presentation $\cJ=[\ut^\ud]$ and $N>0$ a natural number such that $w_i=N/d_i\in\NN$ for $1\le i\le n$. Let $H'$ be the strict transform of $H=V(t_i)$ under $X'=\NBl_\cJ(X)\to X$. Then $H'=\NBl_{\cJ|_H}(H)^\nor$ and the ideals of $H$ and $H'$ are related by $\cI_{H'}=\cJ'^{1/N}\cI_H$, where $\cJ'=\cJ\cO_{X'}$.
\end{lem}


\subsubsection{Complements}
There is a very natural alternative way to do the same construction, which was chosen in \cite{ATW-weighted} but will not be used in this paper: one associates to a center (or a $\QQ$-ideal) $\cJ$ the Rees algebra $\cR_\cJ=\oplus_{n=0}^\infty\cR_n$, where $\cR_n=(\cJ^n)$. Then the (normalized) blowing up along $\cJ$ is defined as the stacky proj $\cProj_X(\cR_\cJ)$ of the graded $\cO_X$-algebra $\cR_\cJ$ whose definition imitates the usual $\Proj_X$ but uses the stack theoretic quotient by $\GG_m$. It is easy to see that $\cR_\cJ$ is normal and one indeed obtains an equivalent definition (and the same charts), and we refer to \cite[\S3]{Quek-Rydh} or \cite[\S3]{ATW-weighted} for details.




\section{Canonical centers via derivations}\label{canonicalsec}

\subsection{Derivations}\label{derivsec}

\subsubsection{Notation}\label{notationsec}
Given a regular scheme $X$ we denote by $\cD_X=\cD_{X/\QQ}$ the sheaf of all absolute derivations (over $\QQ$). We will often work with subsheaves of derivations $\cF\subseteq\cD_X$ and explicitly state which properties are needed in each claim. By $\cF^{\le i}$ we denote the sheaf of differential operators of order at most $i$ generated by $\cF$. Applying it to an ideal $\cI$ we obtain the $i$-th $\cF$-derivation $\cF^{\le i}(\cI)$.

\subsubsection{Separating modules of derivations}
An $\cO_X$-submodule $\cF\subseteq\cD_X$ is called {\em separating} at $x$ if $\cF_x$ contains enough derivations to separate regular coordinates $t_1\.t_n\in\cO_x$. Equivalently, the map $\cF_x\to(m_x/m_x^2)'$ is onto, that is, $\cF$ generates the tangent space at $x$.

\begin{lem}\label{sepderlem}
Let $X$ be a regular scheme, $t_1\.t_n$ regular parameters at a point $x\in X$ and $\cF\subseteq\cD_X$ a module of derivations. Then the following conditions are equivalent:

(i) $\cF$ is separating at $x$.

(ii) there exists derivations $\partial_1\.\partial_n\in\cF_x$ such that $\partial_i(t_j)\in\delta_{ij}+m_x$.

(iii) there exists derivations $\partial_1\.\partial_n\in\cF_x$ such that $\partial_i(t_j)=\delta_{ij}$.
\end{lem}
\begin{proof}
The first two properties are obviously equivalent. It remains to deduce (iii) from (ii). Starting with a family $(\partial_1\.\partial_n)$ as in (ii) we want to find linear combinations $\partial'_j=\sum_i a_{ij}\partial_i$ with $a_{ij}\in\cO_x$ which satisfy (iii). This amounts to solving a linear system $\sum_i a_{ij}\partial_i(t_k)=\delta_{jk}$, that is, to inverting the matrix $(\partial_i(t_k))$. By our assumption this matrix is the unit matrix modulo $m_x$, hence it is invertible.
\end{proof}

\begin{rem}
(i) We will denote the derivations in (iii) as $\partial_i=\partial_{t_i}$, though this property does not define them uniquely. Formally locally, such a set defines the choice of a field of definition $k=k(x)\into \hatcO_x$, on which it vanishes, and once this field is fixed, these derivations are indeed the usual $\partial_{t_1}\.\partial_{t_n}$ with respect to the presentation $k\llbracket\ut\rrbracket=\hatcO_x$.

(ii) Of course, these conditions have been already considered a long time ago. For example, see \cite[Theorem~30.6]{Matsumura-ringtheory} and the weak Jacobian condition after it.
\end{rem}

\subsubsection{Schemes with enough derivations}
We say that a regular scheme $X$ {\em has enough derivations} if $\cD_X$ is separating at any point $x\in X$.

\begin{lem}\label{minifoldlem}
If $X$ has enough derivations and $Y$ is a regular scheme of finite type over $X$, then $Y$ has enough derivations too.
\end{lem}
\begin{proof}
Locally we can realize $Y$ as a closed subscheme in $Z=\bbA^n_X$. First, $Z=\Spec(\cO_X[t_1\.t_n])$ has enough derivations. Indeed, the submodule of $\cD_Y$ generated by the pullback of $\cD_X$ and the $\cO_X$-derivations $\partial_{t_i}$ generate a separating module of derivations on $Z$. Locally on $Z$ we can choose parameters $x_1\.x_m$ so that $Y$ is given by the vanishing of $x_1\.x_d$. Choose $\partial_i$ as in Lemma~\ref{sepderlem}(iii). Then $\partial_{d+1}\.\partial_m$ restrict to derivations on $Y$ and generate a separating family on it.
\end{proof}

\begin{exam}\label{formalexam}
(i) Most regular schemes of a geometric origin have enough derivations -- varieties or schemes of the form $\Spec(A)$, where $A=\Gamma(\cO_X)$ and $X$ is a regular formal or affinoid variety, or a smooth Stein compact. However, only the local formal case is critical for our sequel results: each scheme $\Spec(k\llbracket t_1\.t_n\rrbracket)$ has enough derivations (e.g. by \cite[Theorem~30.8]{Matsumura-ringtheory}).

(ii) A simple example of an excellent DVR in characteristic zero, which does not have enough derivations is described in \cite[Example~2.3.5(ii)]{Temkin-survey}: take any field $k$ of characteristic zero, in $k((t))$ consider a subfield $k(t,x)$ with the induced discrete valuation, where $x=\sum_{i=0}^\infty c_it^i$ and its derivation $\partial_t(x)$ are algebraically independent over $k(t)$, and take $\cO=k(x,y)\cap k\llbracket t\rrbracket$ to be the corresponding valuation ring. It is a non-divisorial DVR on $k(t,x)$, and one easily checks that $\Der(\cO,\cO)=0$. On the other hand, a general theory implies that any DVR containing $\QQ$ is excellent.
\end{exam}

\subsubsection{Lifting derivations}\label{logregmor}
Naturally we will want to lift derivations on $X$ to $Y$, but there might be an obstacle as one only has the following exact sequence: $$0\to\cD_{Y/X}\to\cD_Y\stackrel{\psi_f}\to\Der_X(\cO_X,\cO_Y).$$ If $\psi_f$ is surjective, then we say that $f$ {\em lifts derivations}. This property will be needed in order to prove various functoriality results involving derivations.

\begin{exam}\label{liftexam}
The following morphisms lift derivations:
\begin{itemize}
\item[(i)] Smooth morphism: this follows from the first exact sequence of derivations.
\item[(ii)] A localization $X_x=\Spec(\cO_x)\into X$ or a formal completion $\hatX_x=\Spec(\hatcO_x)\into X$: by continuity of derivations.
\item[(iii)] A morphism $\Spec(l\llbracket t_1\.t_{m+n}\rrbracket)\to\Spec(k\llbracket t_1\.t_n\rrbracket)$: just extend a derivation by zero to $t_{n+1}\.t_{n+m}$ and a transcendence basis $S$ of $l$ over $k$, and then lift it uniquely through the algebraic extension $l/k(S)$.
\end{itemize}
\end{exam}

Note that for any $\cO_X$-submodule $\cF\subseteq\cD_X$ the module $f^*\cF$ is a submodule of $\Der_X(\cO_X,\cO_Y)$ and we will, in fact, only need to lift elements of $f^*\cF$ to $\cD_Y$.

\begin{lem}\label{derideal}
Assume that $f\:X'\to X$ is regular, $\cF\subseteq\cD_X$ and $\cF'\subseteq\cD_{X'}$ are submodules such that $\cF'$ is mapped onto $f^*(\cF)$, and $\cI$ an ideal on $X$ with $\cI'=\cI\cO_{X'}$. Then $\cF^{\le i}(\cI)\cO_{X'}=\cF'^{\le i}(\cI')$.
\end{lem}
\begin{proof}
Since $\cF^{\le i}(\cF^{\le j}(\cI))=\cF^{\le i+j}(\cI)$, we can assume by induction that $i=1$. Any element of $\cF'^{\le 1}(\cI')$ is generated by elements of the form $b\partial(ag)=b\partial(a)g+ab\partial(g)$ with $g\in\cI$ and $a,b\in\cO_{X'}$, hence $\cF'^{\le 1}(\cI')\subseteq \cF^{\le 1}(\cI)\cO_{X'}$. The opposite inclusion holds because any derivation $\partial\in\cF$ lifts to $\cF'$.
\end{proof}

\subsection{Maximal contacts and coefficients ideals}\label{coefsec}

\subsubsection{The order and derivations}
Order of an ideal is defined formally locally and in general its global behaviour can be nasty. It is easy to see that it is upper-semicontinuous on arbitrary qe schemes, but we restrict for now to the case when $X$ has enough derivations.

\begin{lem}\label{derlogord}
Assume that $X$ is a regular scheme, $\cI$ is an ideal on $X$ and $\cF\subseteq\cD_X$ is separating, then

(i) $\ord_X(\cI)$ is the minimal number $d\in\NN$ such that $\cF^{\le d}(\cI)=\cO_X$ (and $\ord_X(\cI)=\infty$ if no such number exists). In particular, $\ord_x(\cI)$ is the minimal $d$ such that $\cF^{\le d}(\cI_x)=\cO_x$.

(ii) If $\ord_X(\cI)=d$, then the closed set $V(\cF^{\le d-1}(\cI))$ is the maximality locus of $\ord_\cI$. In particular, $\ord_\cI$ is upper semicontinuous.
\end{lem}
\begin{proof}
The first claim follows from the formal local description of the order and the facts that the derivations on $X$ extend to $\hatcO_x$ by continuity and for any $t_i$ there exists $\partial_i\in\cF$ taking it to a unit. The second claim follows from the local part of the first one.
\end{proof}

\subsubsection{Coefficients ideal}
If the order $d=\ord(\cI)$ is finite and positive, then the {\em $\cF$-coefficients ideal} of $\cI$ is defined as usual (and we do not need in this paper a more refined homogenized version): $\cC_\cF(\cI)=\sum_{i=0}^{d-1}(\cF^{(\le i)}(\cI))^{\frac{d!}{d-i}}$. If $\cF=\cD_X$, then we will use the notation $\cC_X(\cI)$.

\begin{lem}\label{coefflem}
Let $X$ be a regular scheme with a separating sheaf of derivations $\cF\subseteq\cD_X$, let $\cI$ be an ideal on $X$ with $0<\ord(\cI)<\infty$ and let $\cC=\cC_\cF(\cI)$, then

(i) The maximality locus of $\ord_\cI$ coincides with $V(\cC)$. Namely, $x\in V(\cC)$ if and only if $\ord_x(\cI)=\ord_X(\cI)$.

(ii) If $f\:X'\to X$ is a regular morphism, $\cF'\subseteq\cD_{X'}$ is mapped onto $f^*\cF$ and $\cI'=\cI\cO_{X'}$, then $\cC_{\cF'}(\cI')=\cC\cO_{X'}$.
\end{lem}
\begin{proof}
The first claim is proved precisely as Lemma~\ref{derlogord}(i). The second claim follows from Lemma~\ref{derideal}.
\end{proof}

\subsubsection{Local maximal contact}
If $0<d=\ord_x(\cI)<\infty$, then by a {\em maximal contact to $\cI$ at $x\in X$} we mean any parameter $t\in\cD_X^{(\le d-1)}(\cI_x)$ as well as the subscheme $H=V(t)$ it defines locally at $x$.

\subsubsection{Global maximal contact}
By a {\em maximal contact to $\cI$} we mean an element $t\in\cO_X(X)$ and a hupersurface $H=V(t)$ such that $H$ is a maximal contact to $\cI$ at any point $x$ where the order is maximal: $\ord_x(\cI)=\ord_X(\cI)$. In particular, $H$ contains the maximality locus of $\ord_\cI$. We warn the reader that $H$ may (and usually does) contain points $y$ where it is not a maximal contact (or even $y\notin V(\cI)$). This notion agrees with the local one:

\begin{lem}\label{localcontact}
Keep the above notation and assume that $X$ has enough derivations and $H=V(t)$ is a maximal contact to $\cI$ at $x$. Then there exists a neighborhood $U$ of $x$ such that $H|_U$ is a maximal contact to $\cI|_U$.
\end{lem}
\begin{proof}
We have that $t=\sum a_i\partial_i(b_i)$ in $\cO_{X,x}$, where $a_i\in\cO_x$, $b_i\in\cI_x$ and $\partial_i$ are differential operators of order at most $d-1$. In addition, there exists $\partial\in\cD_x$ such that $\partial(t)=1$. Shrinking $U$ we can assume the same formulas hold in $U$ and $b_i\in\cI(U)$. In particular, $H|_U$ is regular. Since each $b_i$ is of order at least $d$ at any $y\in U$ with $\ord_y(\cI)=d$, we have that $$\ord_y(t)\ge\min_i(\ord_y(\partial_i(b_i)))\ge d-(d-1)=1.$$ Thus $y\in H$, and since $H$ is regular at $y$ and $t\in\cD_X^{(\le d-1)}(\cI)$, we have that $H$ is a maximal contact at $y$.
\end{proof}

\subsubsection{Restriction to maximal contact}
In the theory of maximal contact one wants to restrict an ideal $\cI$ to a maximal contact hypersurface. In order not to loose information one has to take into account all derivations of $\cI$ or the whole coefficients ideal. Here is the main incarnation of the principle that when restricting the coefficients ideal of $\cI$ to a maximal contact one keeps the essential information about $\cI$. In the sequel, this key lemma will easily imply existence of the canonical center of an ideal. We say that a center $\cJ$ is {\em $\cI$-admissible} if $\cI\subseteq\cJ$.

\begin{lem}\label{restrlem}
Let $X$ be a regular scheme, $\cF\subseteq\cD_X$ a separating sheaf of derivations, $x\in X$ a point, $\cI$ an ideal on $X$ such that $0<d=\ord_x(\cI)<\infty$ and $t\in\cO_X$ with $H=V(t)$ a maximal contact to $\cI$ at $x$. Then the order of an $\cI$-admissible center does not exceed $d$, and a center $\cJ$ of order $d$ is $\cI$-admissible if and only if $t^d\in\cJ$ and $\cC_\cF(\cI)|_H\subseteq\cJ^{(d-1)!}|_H$.
\end{lem}
\begin{proof}
Of course, $d=\ord_x(\cI)\ge\ord_x(\cJ)$ for any $\cI$-admissible center $\cJ$. Now, assume that $\cJ$ is a center such that $\ord_x(\cJ)=d$ and let us prove the equivalence. Assume first that $\cI\subseteq\cJ$. Then $t\in\cD^{(\le d-1)}(\cI)\subseteq\cD^{(\le d-1)}(\cJ)$ and hence $t^d\in\cJ$ by Lemma~\ref{diffcenter}(i). In addition, $\cC_\cF(\cI)|_H\subseteq\cC_\cF(\cJ)|_H\subseteq\cJ^{(d-1)!}|_H$ by Lemma~\ref{diffcenter}(ii).

Conversely, assume that $t^d\in\cJ$ and $\cC_\cF(\cI)|_H\subseteq\cJ^{(d-1)!}|_H$. Fix $\partial\in\cF$ with $\partial(t)=1$. Let $f\in\cI$ be any element. Then $(\partial^i(f))^{\frac{d!}{d-i}}|_H\in\cC_\cF(\cI)|_H\subseteq\cJ^{(d-1)!}|_H$, whenever $0\le i\le d-1$, hence $\partial^i(f)|_H\in\cJ^{\frac{d-i}d}|_H$, and thus $f\in\cJ$ by Lemma~\ref{diffcenter}(iii).
\end{proof}

\subsection{Canonical centers via maximal contact}

\subsubsection{Maximal and canonical centers}
Given an ideal $\cI$ on a regular scheme $X$, an $\cI$-admissible center $\cJ$ is called {\em maximal} if it has the maximal invariant among all $\cI$-admissible centers: $\inv(\cJ)>\inv(\cJ')$ for any $\cI$-admissible center $\cJ'\neq\cJ$. If, in addition, for any open subscheme $U\subseteq X$ with $U\cap V(\cJ)\neq\emptyset$ one has that $\cJ_U=\cJ|_U$ is the maximal $\cI_U$-admissible center, then we say that $\cJ$ is the {\em canonical} $\cI$-admissible center and use the notation $\cJ=\cJ(\cI)$. Thus, canonical centers are the maximal centers whose maximality is preserved by localizations.

\begin{rem}\label{maxcontactrem}
(i) We stress that the maximality condition in the definition only refers to the multiorder and the support of the centers, but not to the inclusion. For example, if $\cI=(x^2+xy^2)$ on $X=\Spec(k[x,y])$, then both $(x)$ and $[x^2,y^4]=(x^2,xy^2,y^4)$ are maximal $\cI$-admissible centers with respect to the inclusion, but $[x^2,y^4]$ is the maximal center at the origin.

(ii) Maximal centers do not exist in general. A simple example is obtained when the maximality locus of the order is not closed, see \cite[\S2.1.2]{dream_qe}, though there are other examples, see  \cite[Example~2.1.4]{dream_qe}. Nevertheless, we will see that canonical centers exist in a wide range of situations (for all excellent schemes, as will be proved in the sequel paper) and are compatible with arbitrary regular morphisms (even with non-excellent sources). The latter happens because even when $X'$ is pathological (e.g. non-excellent), the ideal $\cI'$ is of a special form.

(iii) In view of (ii), maximality alone implies canonicity on excellent schemes. We do not know if this is so on pathological schemes and instead of studying this (not so interesting) question simply impose the local maximality condition in the definition. It is used to glue local centers together.
\end{rem}

\subsubsection{The synchronization trick}\label{syncrtrick}
If $\cI=(\cI_1,\cI_2)$ is an ideal on a regular scheme $X=X_1\coprod X_2$ and $\cJ_1,\cJ_2$ are the canonical centers of $\cI_1, \cI_2$, then the canonical center of $\cI$ is either $(\cI_1,\cO_{X_2})$, or $(\cO_{X_1},\cI_2)$, or $(\cI_1,\cI_2)$ -- depending on whether $\mord(\cJ_1)>\mord(\cJ_2)$, or $\mord(\cJ_1)<\mord(\cJ_2)$, or $\mord(\cJ_1)=\mord(\cJ_2)$. We call this obvious but useful observation the {\em synchronization trick} because it is a close relative of synchronization in the classical resolution, see \cite[Remark~2.3.4]{non-embedded}.

\subsubsection{Multiorder of an ideal}
If $\cI$ possesses a canonical center $\cJ=\cJ(\cI)$, we define its {\em multiorder} to be $\mord_X(\cI)=\mord(\cJ)$.

\begin{lem}\label{canonicallem}
Assume that $X$ is a regular scheme with an ideal $\cI$ of order $d$ which possesses a canonical center $\cJ$. Then $\cJ^m$ is the canonical center of $\cI^m$ for any $m\in\NN$, and $\cJ^{(d-1)!}$ is the canonical center of $\cC=\cC_X(\cI)$. In particular, $$m\cdot\mord_X(\cI)=\mord_X(\cI^m)\ \  {\rm and}\ \  (d-1)!\cdot\mord_X(\cI)=\mord_X(\cC).$$
\end{lem}
\begin{proof}
The claim about $\cI^m$ follows from the observation that a center is $\cI$-admissible if and only if its $m$-th power is $\cI^m$-admissible. Furthermore, $\cC\subseteq\cC_X(\cJ)\subseteq\cJ^{(d-1)!}$ by Lemma~\ref{diffcenter}(ii), hence $\cJ^{(d-1)!}$ is the canonical center of both ideals $\cI^{(d-1)!}\subseteq\cC$.
\end{proof}

\subsubsection{The multiorder function}\label{mordsec}
If a point $x\in V(\cI)$ lies in an open subscheme $U$ such that $\cI_U$ possesses a canonical center $\cJ_U$ such that $x\in V(\cJ_U)$, then we define the {\em multiorder} of $\cI$ at $x$ to be $\mord_x(\cI)=\mord(\cJ_U)$, and this is independent of the choice of a neighborhood by the canonicity property. If this condition is satisfied for any $x\in V(\cI)$, we say that $\cI$ {\em admits canonical stratification} and obtain a {\em multiorder function} $\mord_\cI\:X\to\cQ$ of $\cI$ which sends $x$ to $\mord_x(\cI)$.

\begin{lem}\label{mordlem}
Assume that $X$ is a regular scheme and an ideal $\cI$ on $X$ possesses a canonical stratification. Then the following assertions hold:

(i) $\cI$ possesses a canonical center $\cJ=\cJ(\cI)$.

(ii) The function $\mord_\cI$ is upper semicontinuous and attains the maximum along $V(\cJ)$. In addition, $$\mord_X(\cI)=\mord(\cJ)=\max_{x\in X}\mord_x(\cI).$$
\end{lem}
\begin{proof}
(i) Choose an open cover $X=\cup_{i=1}^nX_i$ such that each $\cI_i=\cI|_{X_i}$ possesses a canonical center $\cJ_i$. Let $\ud=\max_i(\mord_{X_i}(\cJ_i))$, $X'=\coprod_{i=1}^nX_i$ and $\cI'=\cI\cO_{X'}$. By the synchronization trick, $\cI'$ possesses a canonical center, which is obtained by keeping all $\cJ_i$ with $\mord(\cJ_i)=\ud$ and replacing each other $\cJ_i$ by $\cO_{X_i}$. Its pullbacks to $X'\times_XX'$ coincide and hence $\cJ'$ descends to an $\cI$-admissible center $\cJ$ whose maximality and universality are checked by descent as well.

(ii) Of course, $\mord_x(\cI)=\ud$ for $x\in V(\cJ)$. Assume that $y\notin V(\cJ)$ and let $U$ be an open neighborhood of $y$ in $X\setminus V(\cJ)$ such that $\cI_U=\cI|_U$ possesses a canonical center $\cJ_U$ and $y\in V(\cJ_U)$. Since $(\cJ,\cO_U)$ is the pullback of $\cJ$ to $X\coprod U$, we have that $\ud=\mord(\cJ)>\mord(\cJ_U)=\ord_y(\cI)$. Thus, the closed set $V(\cJ)$ is the maximality locus of $\mord_\cI$. In addition, any $y\notin V(\cJ)$ has a neighborhood $U$ as above and in this neighborhood $y$ lies in the closed set $V(\cJ_U)$, which is the maximality locus of $\mord_{\cI_U}=(\mord_\cI)|_U$. So, $\mord_\cI$ is upper semicontinuous.
\end{proof}

\begin{rem}
The canonical stratification alluded to in the terminology consists of the canonical centers $\cJ_i=\cJ(\cI|_{X_i})$ on the open subschemes $X_i=X_{i-1}\setminus V(\cJ_{i-1})$, where $X_0=X$. It follows from the lemma that it exists, provides a stratification of $V(\cI)$ and is compatible with localizations.
\end{rem}

\subsubsection{Maximal contacts flag}
By a {\em regular flag} of {\em length} $n$ in $X$ we mean a sequence $H_0\.H_n$ such that $H_0=X$ and $H_{i+1}$ is a regular divisor in $H_i$. Given an ideal $\cI$ and a regular flag we define the {\em restricted coefficients ideals} $\cI_i\subseteq\cO_{H_i}$ via the following recursive definition: $\cI_0=\cI$ and $\cI_{i+1}=\cC(\cI_i)|_{H_{i+1}}$ for $0\le i\le n-1$.

We will only need the above construction in conjunction with the maximal contact property. Namely, we say that a regular flag $H_0\.H_n$ is a {\em maximal contacts flag} to $\cI$ if each $H_i$ with $1\le i\le n$ is a maximal contact to $\cI_{i-1}$ and $\cI_n=0$. In the same manner a family $t_1\.t_n\in\cO_X$ is called a {\em maximal contacts flag} to $\cI$ if the induced flag of subschemes $H_i=V(t_1\.t_i)$ is a maximal contacts flag to $\cI$.

\begin{lem}\label{flagexistlem}
If $X$ has enough derivations, $\cI$ is an ideal on $X$ and $x\in V(\cI)$ is a point, then there exists a neighborhood $U$ of $x$ such that $\cI|_U$ possesses a maximal contacts flag $H_0\.H_n$ and $x\in H_n$.
\end{lem}
\begin{proof}
One uses induction on the dimension. If $\cI_x\neq 0$, then locally at $x$ there exists a maximal contact element $t\in\cD^{(\le d-1)}(\cI)$ and one can take $t_1=t$, $H_1=V(t_1)$, $\cI_1=\cC(\cI)|_H$ and apply the induction assumption to $H_1$ and $\cI_1$.
\end{proof}

\subsubsection{The associated center and invariant}
Given a maximal contacts flag $\ut=(t_1\.t_n)$ with restricted coefficients ideals $\cI_i\subseteq\cO_{H_i}$ we define {\em the associated center} $\cJ=\cJ_\ut(\cI)$ with the {\em associated invariant} $\inv_\ut(\cI)=\inv(\cJ)$ as follows. For $1\le i\le n$ let $e_i=\ord(\cI_{i-1})$ and $d_i=e_i/\prod_{j=1}^{i-1}(e_j-1)!$, and set $\cJ=[\ut^\ud]$ and $\mord_\ut(\cI)=\mord(\cJ)=\ud$. Since the degrees $e_i$ are integral, the tuple $\ud$ of normalized orders lies in the subset $\cQ_1\subset\cQ$ consisting of tuples $(d_1\.d_n)$ such that the numbers $e_i$ iteratively defined by $e_i=d_i\prod_{j=1}^{i-1}(e_j-1)!$ are integral (in particular, $d_1=e_1$ is integral). Note that the set $\cQ_1$ is well ordered due to a bound on the denominator of each $d_i$ which only depends on $d_1\.d_{i-1}$.

The iterative nature of our construction is encoded in the following result, which will be useful for inductive proofs and whose proof is just unravelling of definitions.

\begin{lem}\label{indlem}
Keep the above notation, then $\cJ_\ut(\cI)=[t_1^{d_1},\cJ_1^{1/(d_1-1)!}]$, where $\cJ_1=\cJ_{(t_2\.t_d)}(\cI_1)$.
\end{lem}

\begin{rem}\label{wellordered}
(o) Note that $d_{n+1}=\infty$ because $\cI_n=0$, fitting our convention for the full notation $(d_1\.d_n,\infty)$ for tuples. This makes an even deeper sense in the logarithmic analogues, e.g. see \cite{Quek}, where $\cI_n$ is a monomial ideal, so its log order is infinite, but $\cI_n$ is not necessarily zero.

(i) One can remove the annoying scaling factors in the lemma by working with the iterated restrictions of the normalized coefficients ideal $[\cC(\cI)]^{1/(d-1)!}$. On the one hand, this is more natural, but on the other hand, in order to use the differential calculus one will have to raise them to appropriate powers, so it seems that this does not give anything beyond minor notational convenience.
\end{rem}

\subsubsection{Independence of the maximal contact}
Now we can establish the first main result of the theory -- existence and functoriality of the canonical center, when there are enough derivations. It is locally constructed as the associated center and the main ingredient is to prove independence of maximal contacts. We formulate the theorem in the generality of stacks, so note that the definitions of canonical centers and multiorder function straightforwardly extend to ideals on regular stacks. Also, we say that a regular stack {\em has enough derivations} if it possesses a smooth cover by a scheme with enough derivations.

\begin{theor}\label{localcanonicalcenter}
Assume that $X$ is regular stack with enough derivations and $\cI$ is an ideal on $X$, then

(i) The canonical center $\cJ=\cJ(\cI)$ exists and, moreover, $\cI$ possesses a canonical stratification.

(ii) Let $f\:X'\to X$ be a regular morphism. If $f(X')\cap V(\cJ)\neq\emptyset$, then $\cJ'=\cJ\cO_{X'}$ is the canonical center of $\cI'=\cI\cO_{X'}$. Moreover, in the general case one has that $\mord_{\cI'}=\mord_\cI\circ f$ and the canonical stratification of $\cI'$ is the pullback of the canonical stratification of $\cI$.

(iii) If $X$ is a scheme and $\cI$ possesses a maximal contacts flag $\ut=(t_1\.t_n)$, then the associated center is the canonical one: $\cJ=\cJ_\ut(\cI)$.
\end{theor}
\begin{proof}
First, we note that it suffices to prove all assertions of the theorem for schemes, because then the case of stacks follows from the functoriality with respect to smooth morphisms asserted in (ii). So, in the sequel $X$ and $X'$ are assumed to be schemes.

(i) We claim that this follows from (iii). Indeed, by Lemma~\ref{flagexistlem} any $x\in V(\cJ)$ possesses a neighborhood $U$ such that $\cI_U$ possesses a maximal contacts flag at $x$, and then (iii) implies that the canonical center of $\cI_U$ exists and its support contains $x$.

(iii) By induction we can assume that the assertion holds true when the length of the maximal contacts flag does not exceed $n-1$, and we will see that the induction step is essentially Lemma~\ref{restrlem}. Set $t=t_1$, $H=H_1$, $\cJ_\ut(\cI)=[\ut^\ud]$, $V=V(\cJ_\ut(\cI))$ and $\ut'=(t_2\.t_n)$, $\ud'=(d_2\.d_n)$. Also, set $d=d_1$ and recall that $d=\ord(\cI)=\ord(\cJ_\ut(\cI))$.

It suffices to prove that $\cJ_\ut(\cI)$ is the maximal $\cI$-admissible center because then the same argument applies to the restriction of $\cI$ and $\cJ_\ut(\cI)$ onto any open $U$ such that $U\cap V\neq\emptyset$. Assume that $\cJ$ is $\cI$-admissible and $\inv(\cJ)\nless(\ud,V)$. We will show that $\cJ=\cJ_\ut(\cI)$ by restricting to $H$, using the induction assumption there and lifting the equality on $H$ by Lemma~\ref{indlem}. Note that $\ord(\cJ)=d$, and hence by Lemma~\ref{restrlem} $t$ is a maximal contact to $\cJ$ and $\cI_1=\cC(\cI)|_H\subseteq\cJ_H^{(d-1)!}$, where we set $\cJ_H=\cJ|_H$. In addition, $\inv(\cJ)=(d,\inv(\cJ_H))$, so $\inv(\cJ_H)\nless(\ud',V|_H)$. Recall that $\cJ_{\ut'}(\cI_1)=[\ut'^{\ud'}]^{(d-1)!}$ by Lemma~\ref{indlem}, and hence $[\ut'^{\ud'}]^{(d-1)!}=\cJ_H^{(d-1)!}$ by the induction assumption applied to $\cI_1$. Therefore $[\ut'^{\ud'}]=\cJ_H$, and by Theorem~\ref{centerindlem}(ii) we necessarily have that $\cJ_\ut(\cI)=[\ut^\ud]=\cJ$.

It remains to show that $\cJ_\ut(\cI)$ itself is $\cI$-admissible. Note that $$\cJ_\ut(\cI)^{(d-1)!}|_H=[\ut^\ud]^{(d-1)!}|_H=[\ut'^{\ud'}]^{(d-1)!}=\cJ_{\ut'}(\cI_1)$$ is $\cI_1$-admissible by the induction assumption, and hence $\cJ_\ut(\cI)$ is $\cI$-admissible by Lemma~\ref{restrlem}.

(ii) It suffices to check that $\cJ'$ is the canonical center of $\cI'$, as the second claim is an easy consequence. Assume first that $X'$ has enough derivations and $f$ lifts derivations. Then the claim follows from the functoriality of all ingredients: maximal contacts and coefficients ideals. Namely, $\ut$ pulls back to a maximal contacts flag $\ut'$ and $\cJ$ pulls back to $\cJ_{\ut'}(\cI')$, which is maximal $\cI'$-admissible by Step 1.

Assume now that $f$ is an arbitrary regular morphism such that $f^{-1}(V(\cJ))\neq\emptyset$, but the claim  fails. Shrinking $X'$ we can assume that $\cJ'$ is not maximal, so let $\cJ''$ be another $\cI'$-admissible center such that $\inv(\cJ'')\nless\inv(\cJ')$. Choose a point $y'\in X'$ such that $\cJ'_{y'}\neq\cJ''_{y'}$. Then $\nu_{\cJ',y'}\neq\nu_{\cJ'',y'}$, and hence the induced centers $\hatcJ'_{y'}$ and $\hatcJ''_{y'}$ on the completion $\hatX'_{y'}=\Spec(\hatcO_{X',y'})$ are different. Choose $x'\in V(\cJ')$ for synchronization and set $\hatX'=\hatX'_{y'}\coprod\hatX'_{x'}$, $\hatcI'=\cI\cO_{\hatX'}$, $\hatcJ'=\cJ'\cO_{\hatX'}$ and $\hatcJ''=\cJ''\cO_{\hatX'}$. Since $\inv(\hatcJ'')\nless\inv(\hatcJ')$ and $\hatcJ'\neq\hatcJ''$ the $\hatcI'$-admissible center $\hatcJ'$ is not maximal.

On the other hand, set $x=f(x')$, $y=f(y')$ and $\hatX=\hatX_y\coprod\hatX_x$. Then $\hatcJ'$ is also the pullback of $\cJ$ along the composition $\hatX'\to\hatX\to X$. Even without checking that the morphism $\hatX\to X$ is regular (which is true but non-obvious), we do know that it lifts derivations by Example~\ref{liftexam}(iii), hence $\ut$ is also the maximal contact to $\hatcI=\cI\cO_{\hatX}$, and the non-trivial (thanks to the choice of $x'$) center $\hatcJ=\cJ\cO_{\hatX}=\cJ_\ut(\hatcI)$ is maximal $\hatcI$-admissible by Step 1. Furthermore, the morphism $\hatX'\to\hatX$ is regular by Lemma~\ref{completelem} and it lifts derivations by Example~\ref{liftexam}(iii). Therefore $\hatcJ'=\hatcJ\cO_{\hatX'}$ is maximal $\hatcI'$-admissible by the first paragraph of Step 2, and the contradiction with the conclusion of the previous paragraph concludes the proof.
\end{proof}

\subsection{Canonical blowing up}
Our next major goal is to prove that the invariant drops after preforming a regularized root blowing up of the canonical center and transforming the ideal.

\subsubsection{Transforms}
If $\cJ$ is an $\cI$-admissible center and $f\:X'\to X$ is a root blowing up along $\cJ$, then by the {\em transform} of $\cI$ under $f$ we mean the pullback of $\cI$ divided by the pullback of $\cJ$ (which is an invertible ideal) and use the notation $\cI'=(\cI\cO_{X'})(\cJ\cO_{X'})^{-1}$. Derivations and coefficients ideal of $\cI$ should be transformed with appropriate powers of $\cJ$, so to avoid confusions we will specify them in the notation. The following lemma is a standard basic computation in any modern resolution argument, but since there is a new subtlety with root blowings up we provide all details.

\begin{lem}\label{coefftransf}
Let $X$ be a regular stack, $\cI$ an ideal on $X$ of order $d$ and $\cJ$ an $\cI$-admissible center of multiorder $(d_1\.d_n)$ such that $d=d_1$. Assume that $N>0$ is a natural number such that $w_i=N/d_i\in\NN$ for $1\le i\le n$ and $X'=\NBl_\cJ(X)$. Then

(i) $\cJ'^{-(d-1)!}\cC_X(\cI)\subseteq\cC_{X'}(\cI')$, where $\cJ'=\cJ\cO_{X'}$.

(ii) If $H$ is a maximal contact to $\cI$ and $\ord(\cI')\ge d$, then $\ord(\cI')=d$ and the strict transform $H'$ of $H$ is a maximal contact to $\cI'$.
\end{lem}
\begin{proof}
Note that $\cJ'^{1/d}=(\cJ'^{1/N})^{w_1}$ is an invertible ideal on $X'$. The main ingredient of the proof is that any derivation $\partial\in\cD_X$ induces a meromorphic derivation on $X'$ which may have a pole along the exceptional divisor $V(\cJ')$, but it is bounded by $\cJ'^{1/d}$, that is, $\cJ'^{1/d}\cD^{\le 1}_X\subseteq\cD_{X'}^{\le 1}$.

This claim can be checked locally, so we assume that $\cJ=[\ut^\ud]$ and work on the $i$-th chart $X'_i$ as described in Example~\ref{chartexam}. In (ii) we can also assume that $H=V(t)$. It suffices to evaluate $\partial$ on the $\cO_X$-generators $s_i=t_i^{1/w_i}$ and $t'_j=t_j/s_i^{w_j}$ with $1\le j\le n$ of $\cO_{X'_i}$, and we have that $$\partial(s_i)=\frac{1}{w_i}s_i^{1-w_i}\partial(t_i)\in\cJ'^{-1/d},$$ $$\partial(t'_j)=s_i^{-w_j}\partial(t_j)-\frac{w_j}{w_i}s_i^{-w_i}t'_j\partial(t_i)\in\cJ'^{-1/d}$$
because $\cJ'|_{X'_i}=(t_i^{d_i})=(s_i^N)$ and $N/d\ge N/d_j=w_j$ for $1\le j\le n$.

The rest is a formal computation with differential operators. The commutation rule $[\partial,f]=\partial(f)$ of differential operators $\partial\in\cD_{X'}$ and $f\in\cO_X$ implies that $\cJ'^{\frac id}\cD^{\le i}_X\subseteq\cD_{X'}^{\le i}$ for any $0\le i\le d-1$ and hence $\cJ'^{-1+\frac id}\cD_X^{\le i}(\cI)\subseteq\cD_{X'}^{\le i}(\cI')$. Claim (i) follows by raising both sides to appropriate powers and summing up. By Lemma~\ref{strictlem} locally on $X$ and $X'$ we have that $H=V(t)$ and $H'=V(t')$, where $t\in\cD_X^{\le d-1}(\cI)$ and $(t)=\cJ'^{1/d}(t')$. But then $t'\in\cJ'^{-1/d}\cD_X^{\le d-1}(\cI)\subseteq\cD_{X'}^{\le d-1}(\cI')$, and this implies the assertion of (ii).
\end{proof}

\subsubsection{The invariant drops}
Now we can prove the second main theorem of the theory of canonical centers. Again, this reduces to a simple inductive computation.

\begin{theor}\label{dropth}
Assume that $X$ is a regular stack with enough derivations and $\cI$ is an ideal on $X$ with nowhere dense non-empty $V(\cI)$. Let $\cJ=\cJ(\cI)$ be the canonical center of $\cI$, let $\ud=(d_1\.d_n)=\mord_X(\cI)$, let $N>0$ be a natural number such that $N/d_i\in\NN$ for $1\le i\le n$, and let $f\:X'=\Bl_{\cJ^{1/N}}(X)\to X$ be the corresponding normalized root blowing up of $X$ along $\cJ$ with the transform $\cI'$ of $\cI$. Then $\mord_{X'}(\cI')<\ud$.
\end{theor}
\begin{proof}
We will use induction on the length $n$ of the multiorder with the vacuous case $n=0$ taken as the induction base. Since canonical center, multiorder and normalized root blowings up are compatible with regular morphisms, we can replace $X$ by its smooth cover and assume that it is a scheme and $\cI$ possesses a maximal contacts flag $\ut$. It suffices to show that the multiorder of the coefficients ideal drops: if $\cC=\cC_X(\cI)$ and $\cC'=\cJ'^{-(d-1)!}\cC_X(\cI)$, where $\cJ'=\cJ\cO_{X'}$, then $\mord_{X'}(\cC')<\mord_X(\cC)$. Indeed, this implies that
$$(d-1)!\ \mord_{X'}(\cI')=\mord_{X'}(\cC_{X'}(\cI'))\le\mord_{X'}(\cC')<\mord_X(\cC)=(d-1)!\ \ud,$$
where $d=d_1=\ord_X(\cI)$, the equalities follow from Lemma~\ref{canonicallem} and the first inequality holds because $\cC'\subseteq \cC_{X'}(\cI')$ by Lemma~\ref{coefftransf}(i).

To prove the claim about $\cC$ we will restrict it to a maximal contact. Set $t=t_1$ for shortness, then $H=V(t)$ is a maximal contact to $\cI$ and hence also to $\cC$. In addition, $\cJ=\cJ_\ut(\cI)=[\ut^\ud]$ by Theorem~\ref{localcanonicalcenter}(iii) and $(d-1)!\ \mord_X(\cI)=(d!,\mord_H(\cC_H))$ by the inductive definition of $\cJ_\ut(\cI)$. In particular, $\mord_H(\cC_H)=(d-1)!\ \ud'$, where $\ud'=(d_2\.d_n)$. Let $H'$ be the strict transform of $H$, then $H'=\sqrt[N]{\Bl}_{\cJ_H}(H)^\nor$ by Lemma~\ref{strictlem}, where we set $\cJ_H=\cJ|_H$. If we view $H'$ as $\sqrt[N(d-1)!]{\Bl}_{\cJ^{(d-1)!}_H}(H)^\nor$ and set $\cJ'_{H'}=\cJ_H\cO_{H'}=\cJ\cO_{H'}$, then $\cC'_{H'}=\cC'|_{H'}=\cJ'^{-(d-1)!}_{H'}\cC_H$ is the transform of $\cC_H$ and hence $\mord_{H'}(\cC'_{H'})<(d-1)!\ \ud'$ by the induction assumption.

Now, assume to the contrary that $\mord_{X'}(\cC')\ge(d-1)!\ \ud$. By Lemma~\ref{coefftransf}(ii), $\ord_{X'}(\cC')=d!$ and $H'$ is a maximal contact to $\cC'$. We claim that this automatically implies that $\mord_{X'}(\cC')\le(d!,\mord_{H'}(\cC'_{H'}))$, and hence $(d-1)!\ \ud'\le\mord_{X'}(\cC'_{H'})$, contradicting the conclusion of the previous paragraph. Indeed, any $\cC'$-admissible center $\cK'$ with $\mord(\cK')\ge(d-1)!\ \ud$ is of order $d!$ and has $H'$ as its maximal contact, therefore by Theorem~\ref{centerindlem}(ii) $\cK'|_{H'}$ is a center of multiorder at least $(d-1)!\ \ud'$, and since $\cK'|_{H'}$ is $\cC'_{H'}$-admissible one obtains that also $\mord_{H'}(\cC'_{H'})\ge(d-1)!\ \ud'$.
\end{proof}

\subsubsection{Canonical blowing up}
For inductive reasons it was convenient to allow various numbers $N$ in the above theorem, but in practice we will always want to take the minimal possible number. So, if $X$ is a regular stack, $\cI$ is an ideal on $X$ possessing a canonical center $\cJ=\cJ(\cI)$ and $\mord(\cJ)=(d_1\.d_n)$, then by the {\em canonical $\cI$-admissible blowing up} we mean the normalized root blowing up $X'=\Bl_{\cJ^{1/N}}(X)\to X$, where $N$ is the minimal positive integral number such that $N/d_i\in\NN$ for $1\le i\le n$.

\subsection{Applications}\label{appsec}
We conclude the paper with a brief indication of the consequences of our main results \ref{localcanonicalcenter} and \ref{dropth}. We do not expand much because this is done precisely as in \cite{ATW-weighted}, and in the sequel paper \cite{dream_qe} we will first strengthen these two results to general excellent schemes and then deduce the same applications in the larger generality. So, here we just want to show what can be easily done already at this stage without developing a finer theory of canonical centers on excellent schemes.

\subsubsection{Dream principalization and embedded resolution}\label{dreamsec}
By {\em principalization} of an ideal $\cI$ on a regular stack $X_0$ we mean a sequence of blowings up (normalized, root, etc.) $X_n\to\dots\to X_0=X$ and transforms $\cI_{i+1}=\cI'_i$ such that each $f_i\:X_{i+1}\to X_i$ is $\cI_i$-admissible and $\cI_n=\cO_{X_n}$ is trivial. In particular, the ideal $\cI\cO_{X_n}$ is invertible. If, in addition, each $f_i$ depends only on $(X_i,\cI_i)$, the principalization is called {\em dream} or memoryless.

In the same manner, if one iteratively transforms the closed subscheme $Z=V(\cI)$ by strict transforms, and so $\cI_{Z'}\supseteq\cI'$ and $Z_n$ is regular, then a sequence is called an {\em embedded desingularization} of $Z$ in $X$. Again, the adjective ``dream'' is added if the process is memoryless.

It follows from Theorems \ref{localcanonicalcenter} and \ref{dropth} and the well orderedness of $\cQ_1$ that both dream principalization and embedded resolution are obtained just by applying the canonical blowing up iteratively and transforming the ideal accordingly. This extends \cite[Theorems~6.1.1 and 1.1.1]{ATW-weighted} from varieties to schemes (or stacks) with enough derivations and proves that the obtained constructions are functorial with respect to arbitrary regular morphisms.

\subsubsection{Non-embedded resolution}
The re-embedding principle from \cite[\S8.1]{ATW-weighted} holds for the dream embedded resolution, and one obtains as a corollary non-embedded resolution of qe stacks that smooth-locally can be embedded as a closed subscheme of constant codimension in a regular scheme with enough derivations. Independence of the embedding is checked on the formal level, on which the minimal embedding is essentially unique. This provides a generalization of \cite[\S8.1]{ATW-weighted} to a wide range of excellent stacks.

\subsubsection{Other categories}
Given a result for qe schemes which is functorial with respect to regular morphisms, a standard machinery allows to transfer it to the categories of various analytic spaces and formal schemes, e.g. see \cite[\S5.2]{non-embedded}. Moreover, the excellent rings corresponding to affinoid spaces, Stein compacts and formal varieties (but not arbitrary qe formal schemes, of course) have enough derivations. Therefore, already the results of this paper easily imply that the analogs of Theorems~\ref{localcanonicalcenter} and \ref{dropth} and \cite[Theorems~1.1.1, 6.1.1, 8.1.1]{ATW-weighted} hold for analytic spaces and formal varieties as well.

\subsubsection{The classical algorithm}
Finally, we note that the classical principalization algorithm described in \cite{Bierstone-Milman-funct} extends to the generality of schemes with enough derivations without any changes, and the same argument as was used in the proof of Theorem~\ref{localcanonicalcenter} shows that the obtained method is functorial with respect to all regular morphisms. The method of W{\l}odarczyk, which is based on homogenized coefficients ideals, can also be extended to schemes with enough derivations, but in this case independence of the maximal contact should be checked formally-locally at a point.

\bibliographystyle{amsalpha}
\bibliography{canonical_center}

\end{document}